\documentclass[reqno]{amsart}
\usepackage{a4wide,calrsfs}
\usepackage{amsmath,bbm}
\usepackage{mathtools}
\usepackage{color}
\usepackage{epsfig}
\usepackage{subfigure}
\usepackage{amssymb}
\usepackage{enumerate}
\usepackage{graphicx}
\usepackage{subfigure}
\usepackage{rotating}
\usepackage{stmaryrd}
\usepackage{upgreek}
\usepackage{amsthm}
\usepackage{amsopn}
\usepackage{epstopdf}
\usepackage{cite}
\DeclareMathAlphabet{\mathcal}{OMS}{cmsy}{m}{n}
\DeclarePairedDelimiter{\intpart}{|\![}{]\!|}

\newtheorem{theorem}{Theorem}[section]

\newtheorem{lemma}{Lemma}[section]
\newtheorem{corollar}{Corollary}[section]

\newcommand{\R}{\mathbb{R}}

\newcommand{\bx}{\mathbf{x}}

\newcommand{\bn}{\mathbf{n}}

\allowdisplaybreaks
\numberwithin{equation}{section}

\title[Existence and uniqueness of an inverse problem]
{On the existence and uniqueness of an inverse problem in epidemiology}

\author[A. Coronel]{An{\'\i}bal\ Coronel$^\dag$}
\author[L. Friz]{Luis Friz$^\dag$}
\author[I. Hess]{Ian Hess$^\dag$}
\author[M. Zegarra]{Mar{\'\i}a Zegarra$^\ddag$}

\address{$^\dag$GMA, Departamento de Ciencias B\'asicas, 
Facultad de Ciencias,
Universidad del B\'{\i}o-B\'{\i}o, 
Campus Fernando May, Chill\'{a}n, Chile.}
\email{acoronel@ubiobio.cl,lfriz@ubiobio.cl,ihess@egresados.ubiobio.cl}

\address{$^\ddag$
Facultad  de Matem\'aticas, Universidad 
Nacional Mayor de San Marcos, Lima, Per\'u}
\email{maria\_zegarra@hotmail.com}

\date{\today}

\begin{document}

\begin{abstract}
In this paper we introduce the functional framework
and the necessary conditions
for the well-posedness of an inverse problem arising 
in the mathematical modeling of disease transmission.
The direct problem is given by
an initial boundary value problem for a reaction diffusion
system. The inverse problem consists in the determination
of the  disease and recovery transmission rates from observed
measurement of the direct problem solution at the end time.
The unknowns of the inverse problem are coefficients of
the reaction term. We formulate the inverse problem 
as an optimization problem for an appropriate cost functional.
Then, the existence of solutions of the inverse 
problem is deduced by proving the existence of a minimizer
for the cost functional. Moreover, we establish the uniqueness 
up an additive constant of identification problem. The uniqueness
is  a consequence of the first order
necessary optimality condition and a stability of the
inverse problem unknowns with respect to the observations.

\end{abstract}
\maketitle

\numberwithin{equation}{section}
\newtheorem{thm}{Theorem}[section]
\newtheorem{lem}[thm]{Lemma}
\newtheorem{prop}[thm]{Proposition}
\newtheorem{cor}[thm]{Corollary}
\newtheorem{defn}{Definition}[section]
\newtheorem{conj}{Conjecture}[section]
\newtheorem{exam}{Example}[section]
\newtheorem{rem}{Remark}[section]
\allowdisplaybreaks

\section{Introduction}

The mathematical modeling of disease transmission is an active research area
of mathematical  
biology~\cite{anderson,bacaer,diekmann,akimenko,armbruster,ge2017,ge,koivu,lu,
nwankwo,saadroy,veliov,widder}. Nowadays, there are several approaches used to construct  
the mathematical models in mathematical epidemiology. 
Despite the different kinds of such models, 
and analogously to  biochemical systems, 
we can distinguish five common steps in the processes of modelling~\cite{chou_2009}:  
collection
and analysis of experimental data and information on the specific disease;
 selection of the mathematical theory to be used in the model formulation; 
 the mathematical analysis of well-posedness of the model;  the calibration or
parameter identification of the model; and  the model validation and refinement.
Moreover, we note that the modeling is a cyclical rather than a linear
process: all assumptions made in the
previous steps are reconsidered and refined and 
upon completion of the modeling process. We can improve the model
by introducing new hypotheses, design new experiments,
made predictions and deep the analysis  of each step.
Thus, in particular, we are interested in the analysis of calibration or
parameter identification of the model. To be more precise,
the aim of this paper is to provide a framework to solve the
inverse problem arising in the step of model calibration
by assuming that the mathematical model is an initial boundary 
value problem for a reaction-diffusion system.

Let us precise the mathematical model or the direct problem. 
We consider that the infectious diseases
taken place in a bounded domain  $\Omega\subset \R^{d}$ $(d=1,2,3)$ and
is described by an SIS reaction-diffusion model, where  
the population density of  susceptible and infected individuals 
at time $t$ and location $\bx$ are given by
$S(\bx,t)$ and $I(\bx,t)$,
respectively. The diffusion matrix is assumed to be equals to the identity.
We assume that the infection process is given 
by the interaction of
susceptible and infected densities at the point $\bx$ and time $t$
via  the ``power law'' $\beta(\bx)S^{m}(\bx,t)I^{n}(\bx,t)$,
where $\beta$ is the rate
of disease transmission and $m,n\in ]0,1[$ are some given (fixed) parameters.
The recovery process is represented by $\gamma(\bx)I(\bx,t)$
with $\gamma$ the rate of disease recovery.  
Thus, the direct problem 
is defined as follows: Given the set of functions $\{\beta,\gamma,S_0,I_0\}$
find the functions $S$ and $I$ satisfying the following 
initial boundary value problem
\begin{align}
S_t - \Delta S 
& = -\beta(\bx)S^{m}I^{n}+\gamma(\bx)I, 
&&  \mbox{in } Q_{T}:=\Omega\times [0,T],
\label{eq1:suceptibles}
\\
I_t - \Delta I 
& = \beta(\bx)S^{m}I^{n}-\gamma(\bx)I, 
&& \mbox{in }  Q_{T},
\label{eq2:infectados}
\\
\nabla S\cdot \bn
&=\nabla I\cdot \bn=0,  
&& \mbox{on } \Gamma:=\partial\Omega\times [0,T],
\label{eq3:frontera}
\\
S(\bx,0)&=S_{0}(\bx), && \mbox{in } \Omega,
\label{eq4:inicialS}
\\
I(\bx,0)&=I_{0}(\bx), &&\mbox{in } \Omega,
\label{eq5:inicialI}
\end{align}
where $\partial\Omega$ is the boundary of $\Omega$ and
$\bn$ is the unit exterior normal to $\partial\Omega$.
The boundary conditions~\eqref{eq3:frontera}
and the functions $S_{0}$ and $I_{0}$ models the initial conditions.

The inverse problem consists in the determination
of the rate functions $\beta$ and $\gamma$ in the SIS model
\eqref{eq1:suceptibles}-\eqref{eq5:inicialI} from observed
measurement for $S$ and $I$ at $t=T$ given 
by the functions $S^{obs}$ and $I^{obs}$ defined on $\Omega$.
Then, we can define the inverse problem 
as follows: Given the set of functions $\{S_0,I_0,S^{obs},I^{obs}\}$
defined on $\Omega,$
find the functions $\beta$ and $\gamma$ such that 
the solution $S$ and $I$ of initial boundary value problem
\eqref{eq1:suceptibles}-\eqref{eq5:inicialI}
satisfy the overspecified end condition  $S(\bx,T)=S^{obs}(\bx),$
$I(\bx,T)=I^{obs}(\bx)$ for $\bx\in\Omega$.
Indeed, in order to precise the
analysis of the inverse problem, we consider
an equivalent reformulation as the following optimization problem
\begin{eqnarray}
\inf \, J(\beta,\gamma)
\qquad 
\mbox{subject to $(S_{\beta,\gamma},I_{\beta,\gamma})$ solution of
\eqref{eq1:suceptibles}-\eqref{eq5:inicialI}},
\label{eq9:OptimizacionEquivalencia}
\end{eqnarray}
where 
\begin{eqnarray}
J(\beta ,\gamma):=\frac{1}{2}
\Big[\|S(\cdot,T)-S^{obs}\|^2_{L^2(\Omega)}
+\|I(\cdot,T)-I^{obs}\|^2_{L^2(\Omega)}\Big]
+\frac{\delta}{2}\Big[
\|\nabla\beta\|^2_{L^2(\Omega)}
+\|\nabla\gamma\|^2_{L^2(\Omega)}
\Big]
\label{eq8:Optimizacion}
\end{eqnarray}
is a functional defined on the admissible set  
\begin{align}
U_{ad}(\Omega)&=\mathcal{A}(\Omega) \cap 
\Big[H^{\intpart{d/2}+1}(\Omega)\times H^{\intpart{d/2}+1}(\Omega)\Big],
&&
\label{eq:admissible}\\
\mathcal{A}(\Omega)&=\Big\{(\beta ,\gamma)\in C^{\alpha}(\overline{\Omega})
\times C^{\alpha}(\overline{\Omega})\; 
:\quad\mbox{Ran}(\beta)\subseteq[\underline{b},\overline{b}]\subset ]0,1[,
\nonumber\\ 
&
\hspace{5.2cm}
\mbox{Ran}(\gamma)\subseteq [\underline{r},\overline{r}]\subset ]0,1[, 
\quad \nabla\beta ,\; \nabla \gamma\in L^{2}(\Omega)\Big\},
\end{align}
and for an appropriate $\delta>0$.
Here, 
$H^m(\Omega) $ and $C^{\alpha}(\overline{\Omega})$ denote
the standard Sobolev  and H\"older spaces $W^{m,2}(\Omega)$ and 
$C^{0,\alpha}(\overline{\Omega})$, respectively; and 
$\mbox{Ran}(f)$ denote the range of function $f$. 
The construction of $U_{ad}(\Omega)$ was recently 
developed in \cite{chs:nota_2018} and also
we note that $U_{ad}(\Omega)=\mathcal{A}(\Omega)$ when $d=1$
and coincides with the  admissible set considered by  
Xiang and Liu in \cite{Huili_2015}.  

The main result of this paper is the necessary conditions
for the well-posedness theory of the inverse problem. More precisely,
we prove the following theorem:
\begin{theorem}
\label{teo:inverso}
Let us consider the notation
\begin{align}
 \mathcal{U}(\Omega)&=\Big\{(\beta ,\gamma)\in 
U_{ad}(\Omega)\::\: \|\beta\|_{L^1(\Omega)}
\mbox{ and }\|\gamma\|_{L^1(\Omega)}\mbox{ are constants}\Big\}.
\label{eq:conj_uniq}
\end{align}
Consider that the open bounded and convex set 
$\Omega$ is such that  $\partial\Omega$ is $ C^1$
and the initial conditions $S_0$ and $I_0$ are functions belong 
 to $C^{2,\alpha}(\overline{\Omega})$ and satisfy the inequalities 
\begin{align}
S_0(\bx)\ge 0,
\quad 
I_0(\bx)\ge 0,
\quad
\int_{\Omega}I_0(\bx)d\bx>0,
\quad 
S_0(\bx)+I_0(\bx)\ge \phi_0>0,
\label{eq:hipo_initial_condition}
\end{align}
on $\Omega$,
for some positive constant $\phi_0$.
Moreover assume that the observation functions 
$S^{obs}$ and $I^{obs}$ are functions belong 
to $L^2(\Omega)$.
Then, 
there exists at least one solution of \eqref{eq9:OptimizacionEquivalencia}
and there exist $\Theta\in\mathbb{R}^+$
such that the solution of \eqref{eq9:OptimizacionEquivalencia} is uniquely 
defined, up an additive constant, on 
$\mathcal{U}(\Omega)$ for any regularization parameter~$\delta>\Theta$.
\end{theorem}

On the other hand, we recall that inverse problems in reaction-diffusion equations
and systems have been addressed in the literature of the 
last decades, for instance
\cite{chen_2006,deng_2009,sakthivel_2011,Huili_2015,
friedman_2010,marinova_2014,rahmoun_2014}.
In \cite{chen_2006} the authors study the identification of $q(x)$ in the equation
$u_t=\Delta u+q(x)u$ with Dirichlet boundary condition and
from final measurement data $u(x,T)$. They prove
the existence of solutions and develop a  numerical 
solution of the inverse problem by using an optimization problem.
The authors of \cite{deng_2009}  consider the nonlinear reaction-diffusion 
 equation $u_t=\Delta u+p(x)f(u)$ with 
$f$ a nonlinear function and study the identification of $p$,
getting some results for the existence and the local uniqueness.
Now, in \cite{sakthivel_2011} the authors study  the inverse 
problem for a reaction-diffusion system with a linear reaction term
and obtain existence and local uniqueness of the inverse problem.
More recently, in \cite{Huili_2015} the authors have studied  the one-dimensional 
version of the inverse problem considered in this paper. They obtain a result
for existence and local uniqueness of the solution by assuming that 
the infection process is modeled by a 
frequency-dependent transmission function instead of the power law function.
 Now, the articles
\cite{friedman_2010,marinova_2014,rahmoun_2014} are focused on inverse problems
in epidemic systems, but are of different type to that one considered 
in this paper. Thus, the Theorem~\ref{teo:inverso} is an extension to the multidimensional
global uniqueness framework of the one-dimensional and local uniqueness 
results obtain recently in \cite{Huili_2015}.

The rest of the paper is organized in two sections. In section 2 we present some results
for the direct problem solution, we introduce the adjoint state
and the necessary optimality conditions, and prove a stability result. On 
section~3 we present the proof of Theorem~\ref{teo:inverso}.

\section{Preliminary}

\subsection{Direct problem solution} 

The well-posedness of the direct problem  \eqref{eq1:suceptibles}-\eqref{eq5:inicialI}
is given by the following result.

\begin{theorem}
 \label{teo:direct_problem}
Consider that $\Omega,S_0$ and $I_0$ satisfy
the hypotheses of Theorem~\ref{teo:inverso}. 
If $(\beta,\alpha)\in C^{\alpha}(\overline{\Omega})\times
C^{\alpha}(\overline{\Omega})$,
the initial boundary value problem \eqref{eq1:suceptibles}-\eqref{eq5:inicialI}
admits a unique positive classical solution  $(S,I)$, such that 
$S$ and $I$ are belong to $C^{2+\alpha,1+\alpha/2}(\overline{Q}_T)$
and also
$S$ and $I$ are bounded on~$\overline{Q}_T$, for any given~$T\in \R^+$.
\end{theorem}

The existence and the uniqueness 
can be developed by the Shauder's theory for parabolic
equations~\cite{Ladyzhenskaya,Krylov,liberman}. Meanwhile, the positive  behavior of the 
solution is a consequence of the maximum principle. Indeed, 
if we denote by $N$ the total population, i.e. $N(\bx,t)=S(\bx,t) + I(\bx,t)$.
Then, from the system \eqref{eq1:suceptibles}-\eqref{eq5:inicialI} we
have that $N$ satisfy the following initial boundary value problem
\begin{align*}
N_t - \Delta N &= 0, && \mbox{in }  Q_{T},
\\
\nabla N\cdot \bn&=0, && \mbox{on }\Gamma,
\\
N(\bx,0)&= S_{0}(\bx) + I_{0}(\bx), && \mbox{in } \Omega.
\end{align*}
By the maximum principle of parabolic equations
and the hypothesis \eqref{eq:hipo_initial_condition}
we have that $N(\bx,t)\ge S_{0}(\bx) + I_{0}(\bx)\ge\phi_{0}>0$ 
in~$Q_{T}$.

\begin{corollar}
\label{corolar:bounded}
Consider that $\Omega,S_0$ and $I_0$ satisfy
the hypotheses of Theorem~\ref{teo:direct_problem}. 
If $(\beta,\alpha)\in C^{\alpha}(\overline{\Omega})\times
C^{\alpha}(\overline{\Omega})$
and $(S,I)$ is the solution of the initial boundary value problem 
of \eqref{eq1:suceptibles}-\eqref{eq5:inicialI}, then the estimates
$0<\mathbb{S}_m\le S(\bx,t)\le \mathbb{S}_M,$
and 
$0<\mathbb{I}_m\le I(\bx,t)\le \mathbb{I}_M,$
are valid on $\overline{Q}_T,$ for some strictly positive
constants $\mathbb{S}_m,\mathbb{S}_M,\mathbb{I}_m,$ and $\mathbb{I}_M.$  
\end{corollar}

\subsection{Adjoint System}
Let us consider that $(\bar{\beta},\bar{\gamma})$ is a solution of 
the optimal control problem \eqref{eq9:OptimizacionEquivalencia}
and $(\bar{S},\bar{I})$ is the corresponding solution of 
\eqref{eq1:suceptibles}-\eqref{eq5:inicialI} with
$(\bar{\beta},\bar{\gamma})$ instead of $(\bar{\beta},\bar{\gamma})$.
Then we introduce $(p_{1},p_{2})$ the adjoint variables, i.e., 
the solution of the adjoint system which is given by the following 
backward boundary value problem
\begin{align}
(p_1)_t+\Delta p_1 
&= m \bar{\beta}(\bx)\bar{S}^{m-1}\bar{I}^{n}(p_1-p_2),  
&&\mbox{in } Q_{T},
\label{eq10:adjunto1}
\\
(p_2)_t+\Delta p_2
&= n \bar{\beta}(\bx)\bar{S}^{m}\bar{I}^{n-1}(p_1-p_2)
-\bar{\gamma}(\bx)(p_1-p_2), 
&&\mbox{in } Q_{T},
\label{eq11:adjunto2}
\\
\nabla p_1\cdot \bn
&=\nabla p_2\cdot \bn=0, 
&&\mbox{on }\Gamma ,
\label{eq12:adjunto3}
\\
p_1(\bx,T)&=\bar{S}(\bx,T)-S^{obs}(\bx),  
&&\mbox{in }\Omega ,
\label{eq13:adjunto4}
\\
p_2(\bx,T)&=\bar{I}(\bx,T)-I^{obs}(\bx), 
&&\mbox{in }\Omega .
\label{eq14:adjunto5}
\end{align}
The existence of solutions 
for system \eqref{eq10:adjunto1}-\eqref{eq14:adjunto5}
can be developed by similar arguments to a similar
result presented in~\cite{apreutesei}. Now, for our purpose,
we need some a priori estimates given on the following
result.

\begin{lemma}
Consider that $\Omega,S_0,I_0,S^{obs}$ and $I^{obs}$, satisfy
the hypotheses of Theorem~\ref{teo:inverso}.
Moreover, consider that $(\bar{\beta},\bar{\gamma})\in U_{ad}$ 
is a solution of \eqref{eq9:OptimizacionEquivalencia}, and $(\bar{S},\bar{I})$
is the solution of \eqref{eq1:suceptibles}-\eqref{eq5:inicialI}
with $(\bar{\beta},\bar{\gamma})$ instead of 
$(\beta,\gamma)$. 
Then, the solution of
the adjoint system \eqref{eq10:adjunto1}-\eqref{eq14:adjunto5} 
satisfy the following estimates
\begin{eqnarray}
&&\|p_1(\cdot,t)\|_{L^{2}(\Omega)}^{2} 
+ \|p_2(\cdot,t)\|_{L^{2}(\Omega)}^{2} \leq P_1,
\label{eq:estimatenormaL2}
\\
&&\|p_1 (\cdot,t)\|_{H_0^1(\Omega)}
+ \|p_2(\cdot,t)\|_{H_0^1(\Omega)}\leq P_2,
\label{eq:estimatenormaH01}
\\
&&\|\Delta p_1 (\cdot,t)\|_{L^2(\Omega)}
+ \|\Delta p_2(\cdot,t)\|_{L^2(\Omega)}\leq P_3,
\label{eq:estimatenormaH2}
\\
&&\|p_1 (\cdot,t)\|_{L^{\infty}(\Omega)}\leq P_4,
\quad\|p_2 (\cdot,t)\|_{L^{\infty}(\Omega)}\leq P_5,
\label{eq:estimatenormaLinfinito}
\end{eqnarray}
for $t\in [0,T]$ and some positive constants $P_1,\ldots,P_5$.
\label{lem:boundedp_i}
\end{lemma}

\begin{proof}
Let us consider the change of variable $\tau=T-t$ for $t\in [0,T]$ and 
also consider the notation
\begin{eqnarray*}
w_i(\bx,\tau)=p_i(\bx,T-\tau),\;\; i=1,2,
\quad
S^{*}(\bx,\tau)=\bar{S}(\bx,T-\tau),    
\quad
I^{*}(\bx,\tau)=\bar{I}(\bx,T-\tau). 
\end{eqnarray*}
Then, the adjoint system \eqref{eq10:adjunto1}-\eqref{eq14:adjunto5}
can be rewritten as follows
\begin{align}
(w_1)_\tau-\Delta w_1 
&= -m \bar{\beta}(\bx)({S}^{*})^{m-1}({I}^{*})^{n}(w_1-w_2), 
&&\mbox{in } Q_{T},
\label{eq:nuevoadjunto1}
\\
(w_2)_\tau-\Delta w_2 
&= -n \bar{\beta}(\bx)(S^{*})^{m}(I^{*})^{n-1}(w_1-w_2)
+\bar{\gamma}(\bx)(w_1-w_2),
&&\mbox{in } Q_{T},
\label{eq:nuevoadjunto2}
\\
\nabla w_1\cdot \bn
&=\nabla w_2\cdot \bn=0,  
&&\mbox{on }\Gamma ,
\label{eq:nuevoadjunto3}
\\
w_1(\bx,0)&=\bar{S}(\bx,T)-S^{obs}(\bx), 
\quad
w_2(\bx,0)=\bar{I}(\bx,T)-I^{obs}(\bx),
&&\mbox{in }\Omega.
\label{eq:nuevoadjunto5}
\end{align}
Now, we get for $w_i$ the estimates  of  the form \eqref{eq:estimatenormaL2},
\eqref{eq:estimatenormaH01}, and 
\eqref{eq:estimatenormaLinfinito}. 

In order to prove \eqref{eq:estimatenormaL2}, we
multiply \eqref{eq:nuevoadjunto1} 
by $w_1$ and \eqref{eq:nuevoadjunto2} by $w_2$, 
integrate on $\Omega$ and use the Green formula, to get
\begin{eqnarray*}
\int_{\Omega}(w_1)_\tau w_1\,d\bx 
+ \int_{\Omega}(\nabla w_1)^{2}\, d\bx &=&
-m\int_{\Omega}\bar{\beta}(\bx)(S^{*})^{m-1}(I^{*})^{n}w_1^2\, d\bx
\\
&&+m \int_{\Omega}\bar{\beta}(\bx)(S^{*})^{m-1}(I^{*})^{n}w_1w_2\, d\bx,
\\
\int_{\Omega}(w_2)_\tau w_2\,d\bx
+ \int_{\Omega}(\nabla w_2)^{2}\, d\bx &=&
-\int_{\Omega}\left[ n\bar{\beta}(\bx)(S^{*})^{m}(I^{*})^{n-1}
- \bar{\gamma}(\bx)\right]w_1w_2\, d\bx \nonumber
\\
&&+\int_{\Omega}\left[ n\bar{\beta}(\bx)(S^{*})^{m}(I^{*})^{n-1} 
- \bar{\gamma}(\bx)\right]w_2^{2}\, d\bx ,
\end{eqnarray*}
respectively. Then, adding the equalities, applying the Cauchy 
inequality, rearranging some terms, and applying
the Corollary~\ref{corolar:bounded},  we can deduce the following estimate
\begin{eqnarray}
&&\frac{1}{2}\frac{d}{d\tau}
\Big(\|w_1(\cdot,\tau)\|_{L^2(\Omega)}^2+\|w_2(\cdot,\tau)\|_{L^2(\Omega)}^2\Big)
+\|\nabla w_1(\cdot,\tau)\|_{L^2(\Omega)}^2
+\|\nabla w_2(\cdot,\tau)\|_{L^2(\Omega)}^2
\nonumber\\
&&
\hspace{1cm}
\le \hat{C}
\Big[
\|w_1(\cdot,\tau)\|_{L^2(\Omega)}^2
+\|w_2(\cdot,\tau)\|_{L^2(\Omega)}^2
\Big].
\label{hatCCC_prev}
\end{eqnarray}
with
\begin{align}
\hat{C}=\max\left\{
\frac{3\hat{C}_1+\hat{C}_2}{2},
\frac{\hat{C}_1+3\hat{C}_2}{2}
\right\},
\quad
\hat{C}_1=\overline{b}\: m\:
\mathbb{S}_m^{m-1}
\mathbb{I}_M^n,
\quad
\hat{C}_2=
\overline{b}\: n\:
\mathbb{S}_M^m
\mathbb{I}_n^{n-1}+\overline{r}.
\label{hatCCC}
\end{align}
Then, from \eqref{hatCCC_prev} and the Gronwall inequality, we obtain
\begin{eqnarray}
\|w_1(\cdot,\tau)\|_{L^{2}(\Omega)}^{2} + \|w_2(\cdot,\tau)\|_{L^{2}(\Omega)}^{2} \leq
\left(\|w_1(\cdot,0)\|_{L^{2}(\Omega)}^{2} + \|w_2(\cdot,0)\|_{L^{2}(\Omega)}^{2}\right)
e^{2\hat{C} T},
\label{eq:estimatenormaL2:w}
\end{eqnarray}
which implies \eqref{eq:estimatenormaL2}.

From \eqref{hatCCC_prev} and \eqref{eq:estimatenormaL2:w},
we have that 
\begin{eqnarray*}
\|\nabla w_1(\cdot,\tau)\|_{L^{2}(\Omega)}^{2} 
+ \|\nabla w_2(\cdot,\tau)\|_{L^{2}(\Omega)}^{2} \leq
\hat{C}\:e^{2\hat{C} T}\:
\left(\|w_1(\cdot,0)\|_{L^{2}(\Omega)}^{2} 
+ \|w_2(\cdot,0)\|_{L^{2}(\Omega)}^{2}\right)
\end{eqnarray*}
and by the definition of the norm of $H_0^1(\Omega)$ we deduce
the estimate~\eqref{eq:estimatenormaH01}.

On the other hand, using the  fact that
\begin{eqnarray*}
\int_{\Omega}(w_i)_\tau\Delta w_i\,d\bx 
=-\int_{\Omega}\nabla[(w_i)_\tau]\cdot \nabla w_i\,d\bx
+\int_{\partial\Omega}(w_i)_\tau\nabla(w_i)\cdot\bn \,dS
=-\frac{1}{2}\frac{d}{d\tau}\| w_i(\cdot,\tau) \|_{L_2(\Omega)}^{2},
\end{eqnarray*}
for $i=1,2$. We note that, multiplying \eqref{eq:nuevoadjunto1} by $\Delta w_1$,
multiplying \eqref{eq:nuevoadjunto2} by $\Delta w_2$, 
integrating on~$\Omega$, and adding the results, we deduce that
\begin{eqnarray*}
&&\frac{1}{2}\frac{d}{d\tau}
\Big(\|w_1(\cdot,\tau)\|_{H^1_0(\Omega)}^2+\|w_2(\cdot,\tau)\|_{H^1_0(\Omega)}^2\Big)
+\|\Delta w_1(\cdot,\tau)\|_{L^2(\Omega)}^2
+\|\Delta w_2(\cdot,\tau)\|_{L^2(\Omega)}^2
\\
&&
\hspace{1cm}
\le
\hat{C}
\Big[
\epsilon\|w_1(\cdot,\tau)\|_{L^2(\Omega)}^2
+\frac{1}{4\epsilon}\|\Delta w_1(\cdot,\tau)\|_{L^2(\Omega)}^2
+\epsilon\|w_2(\cdot,\tau)\|_{L^2(\Omega)}^2
+\frac{1}{4\epsilon}\|\Delta w_2(\cdot,\tau)\|_{L^2(\Omega)}^2
\Big]
\end{eqnarray*}
with $\hat{C}$ defined on \eqref{hatCCC} and any $\epsilon>0.$
Then, we have that
\begin{eqnarray*}
&&\frac{1}{2}\frac{d}{d\tau}
\Big(\|w_1(\cdot,\tau)\|_{H^1_0(\Omega)}^2+\|w_2(\cdot,\tau)\|_{H^1_0(\Omega)}^2\Big)
+
\left(1-\frac{\hat{C}}{4\epsilon}\right)
\Big(
\|\Delta w_1(\cdot,\tau)\|_{L^2(\Omega)}^2
+\|\Delta w_2(\cdot,\tau)\|_{L^2(\Omega)}^2
\Big)
\\
&&
\hspace{1cm}
\le
\epsilon\hat{C}
\Big[
\|w_1(\cdot,\tau)\|_{L^2(\Omega)}^2
+\|w_2(\cdot,\tau)\|_{L^2(\Omega)}^2
\Big].
\end{eqnarray*}
Now, by selecting $\epsilon>\hat{C}/4$ and using the estimate 
\eqref{eq:estimatenormaL2:w} we get
\begin{eqnarray*}
&&
\|\Delta w_1(\cdot,\tau)\|_{L^2(\Omega)}^2
+\|\Delta w_2(\cdot,\tau)\|_{L^2(\Omega)}^2
\le
\frac{4\epsilon^2\hat{C}}{4\epsilon-\hat{C}}
e^{2\hat{C} T}
\left(\|w_1(\cdot,0)\|_{L^{2}(\Omega)}^{2}
 + \|w_2(\cdot,0)\|_{L^{2}(\Omega)}^{2}\right),
\end{eqnarray*}
which implies the inequality \eqref{eq:estimatenormaH2}.

From \eqref{eq:estimatenormaL2}, \eqref{eq:estimatenormaH01} and
\eqref{eq:estimatenormaH2}, we have that the norm of $p_1$
and $p_2$ are bounded in the norm of
$H^{2}(\Omega)$. Thus, 
by the standard embedding theorem of  $H^2(\Omega)\subset L^{\infty}(\Omega)$,
we easily deduce 
\eqref{eq:estimatenormaLinfinito} and conclude the proof of the lemma.
\end{proof}

\subsection{Necessary Optimality Conditions.}

\bigskip

\begin{lemma}
Let $(\bar{\beta},\bar{\gamma})$ be the solution of the optimal control problem \eqref{eq9:OptimizacionEquivalencia},
$(\bar{S},\bar{I})$ be the solution of \eqref{eq1:suceptibles}-\eqref{eq5:inicialI} with $(\bar{\beta},\bar{\gamma})$
instead of $(\beta,\gamma)$ and $(p_{1},p_{2})$ the solution of the adjoint system \eqref{eq10:adjunto1}-\eqref{eq14:adjunto5}.
Then, the inequality
\begin{eqnarray}
\iint_{Q_{T}}\left[ \left( \hat{\beta} - \bar{\beta} \right)\bar{S}^{m}\bar{I}^{n} - \left(\hat{\gamma} - \bar{\gamma}\right)\bar{I} \right]
(p_{2}-p_{1})\,d\bx dt \nonumber
\\
+ \delta\int_{\Omega}\left[ \nabla\bar{\beta}\nabla\left(\hat{\beta}- \bar{\beta}\right) +
\nabla\bar{\gamma}\nabla\left(\hat{\gamma}- \bar{\gamma}\right) \right] \,d\bx\geq 0,
\label{teo:necessaryoptimalcondition}
\end{eqnarray}
is satisfied for any $(\hat{\beta},\hat{\gamma})\in U_{ad}$.
\end{lemma}

\begin{proof}
Let us consider an arbitrary pair $(\hat{\beta},\hat{\gamma})\in U_{ad}$
and introduce the notation
\begin{eqnarray*}
(\beta^{\varepsilon},\gamma^{\varepsilon}) 
&=&(1-\varepsilon)(\bar{\beta},\bar{\gamma})
+\varepsilon(\hat{\beta},\hat{\gamma})\in U_{ad},
\\
J_{\varepsilon}&=&J(\beta^{\varepsilon},\gamma^{\varepsilon})
= \frac{1}{2}\int_{\Omega}
\left( \left| S^{\varepsilon}(\bx,t)-
S^{obs}(\bx) \right|^{2} 
+ \left| I^{\varepsilon}(\bx,t)- I^{obs}(\bx) \right|^{2} 
\right)\, d\bx
\\
&&
\hspace{1cm}
+ \frac{\delta}{2}\int_{\Omega}\left( \left|\nabla\beta^{\varepsilon}(\bx)\right|^{2} +
\left|\nabla\gamma^{\varepsilon}(\bx)\right|^{2} \right)\, d\bx ,
\end{eqnarray*}
where $(S^{\varepsilon},I^{\varepsilon})$ is the solution 
of \eqref{eq1:suceptibles}-\eqref{eq5:inicialI} with
$(\beta^{\varepsilon},\gamma^{\varepsilon})$ instead of $(\beta,\gamma)$.
Now, using the hypothesis that $(\bar{\beta},\bar{\gamma})$ 
is an optimal solution of \eqref{eq9:OptimizacionEquivalencia}
and  taking the Fr\'echet derivative of $J_{\varepsilon}$, we
have that
\begin{eqnarray}
\frac{dJ_{\varepsilon}}{d\varepsilon}\Big{|}_{\varepsilon =0}
&= &\int_{\Omega}\left( \left|S^{\varepsilon}(\bx,t)-S^{obs}(\bx) \right|
\frac{\partial S^{\varepsilon}}{\partial\varepsilon}\Big{|}_{\varepsilon =0} +
\left|I^{\varepsilon}(\bx,t)-I^{obs}(\bx) \right|
\frac{\partial I^{\varepsilon}}{\partial\varepsilon}\Big{|}_{\varepsilon =0} \right)\,d\bx  \nonumber
\\
&&+\delta\int_{\Omega}\left[ \nabla\bar{\beta}\nabla\left(\hat{\beta}- \bar{\beta}\right) +
\nabla\bar{\gamma}\nabla\left(\hat{\gamma}- \bar{\gamma}\right) \right] \,d\bx\geq 0,
\label{eq:Frechetderivative}
\end{eqnarray}
where $\partial_{\varepsilon}S^{\varepsilon}$ and  $\partial_{\varepsilon}I^{\varepsilon}$
for $\varepsilon=0$ are calculated by analyzing the sensitivities 
of solutions for \eqref{eq1:suceptibles}-\eqref{eq5:inicialI} 
with respect to perturbations of $(\beta,\gamma)$.

From the definition of $(S^{\varepsilon},I^{\varepsilon})$ and $(\bar{S},\bar{I})$ 
we have that
\begin{align}
(S^{\varepsilon})_t - \Delta S^{\varepsilon} 
& = -\beta^{\varepsilon}(\bx)(S^{\varepsilon})^{m}(I^{\varepsilon})^{n}
+\gamma^{\varepsilon}(\bx)I^{\varepsilon}, 
&&  \mbox{in } Q_{T},
\label{eq1:suceptiblesepsilon}
\\
(I^{\varepsilon})_t - \Delta I^{\varepsilon} 
& = \beta^{\varepsilon}(\bx)(S^{\varepsilon})^{m}(I^{\varepsilon})^{n}
-\gamma^{\varepsilon}(\bx)I^{\varepsilon}, 
&&  \mbox{in } Q_{T},
\label{eq2:infectadosepsilon}
\\
\nabla S^{\varepsilon}\cdot \bn
&=\nabla I^{\varepsilon}\cdot \bn=0,  
&& \mbox{on }\Gamma,
\label{eq3:fronteraepsilon}
\\
S^{\varepsilon}(\bx,0)&=S_{0}(\bx),\quad 
I^{\varepsilon}(\bx,0)=I_{0}(\bx), && \mbox{in }\Omega,
\label{eq4:inicialSepsilon}
\end{align}
and 
\begin{align}
(\bar{S})_t - \Delta \bar{S} 
& = -\bar{\beta}(\bx)(\bar{S})^{m}(\bar{I})^{n}+\bar{\gamma}(\bx)\bar{I}, 
&&  \mbox{in }\in Q_{T},
\label{eq1:suceptiblesbarra}
\\
(\bar{I})_t - \Delta \bar{I} 
& = \bar{\beta}(\bx)(\bar{S})^{m}(\bar{I})^{n}
-\bar{\gamma}(\bx)\bar{I}, 
&& \mbox{in }\in Q_{T},
\label{eq2:infectadosbarra}
\\
\nabla \bar{S}\cdot \bn
&=\nabla \bar{I}\cdot \bn=0,  
&& \mbox{on }\Gamma,
\label{eq3:fronterabarra}
\\
\bar{S}(\bx,0)&=S_{0}(\bx),\quad 
\bar{I}(\bx,0)=I_{0}(\bx), && \mbox{in }\Omega.
\label{eq4:inicialSbarra}
\end{align}
Subtracting the system \eqref{eq1:suceptiblesbarra}-\eqref{eq4:inicialSbarra} 
from \eqref{eq1:suceptiblesepsilon}-\eqref{eq4:inicialSepsilon},
dividing by $\varepsilon$ and using the notation 
$\left( z_{1}^{\varepsilon},z_{2}^{\varepsilon} \right)
=\varepsilon^{-1}\left( S^{\varepsilon}-\bar{S},I^{\varepsilon}-\bar{I} \right)$, 
we deduce the following system
\begin{align}
(z^\varepsilon_1)_t - \Delta z^\varepsilon_1 
& = -\beta^\varepsilon(\bx)
\frac{\left[ (S^{\varepsilon})^{m}-(\bar{S})^{m} \right]}{S^{\varepsilon}-\bar{S}}
(I^{\varepsilon})^{n}z^\varepsilon_1
-\beta^\varepsilon(\bx)(\bar{S})^{m}
\frac{\left[ (I^{\varepsilon})^{n} - (\bar{I})^{n} \right]}
{I^{\varepsilon} - \bar{I}} z^{\varepsilon}_2 
&&
\nonumber\\
&
\quad
-(\hat{\beta}-\bar{\beta})(\bar{S})^{m}(\bar{I})^{n} 
+ \gamma^{\varepsilon}(\bx)z^{\varepsilon}_2 
+ (\hat{\gamma} - \bar{\gamma})\bar{I},
&&  \mbox{in } Q_{T},
\label{eq1:suceptiblesdif}
\\
(z^\varepsilon_2)_t - \Delta z^\varepsilon_2 
& = \beta^\varepsilon(\bx)
\frac{\left[ (S^{\varepsilon})^{m}-(\bar{S})^{m} \right]}{S^{\varepsilon}-\bar{S}}
(I^{\varepsilon})^{n}z^\varepsilon_1
+\beta^\varepsilon(\bx)(\bar{S})^{m}
\frac{\left[ (I^{\varepsilon})^{n} - (\bar{I})^{n} \right]}
{I^{\varepsilon} - \bar{I}} z_{2}^{\varepsilon} 
&&
\nonumber\\
&
\quad
+(\hat{\beta}-\bar{\beta})(\bar{S})^{m}(\bar{I})^{n} 
- \gamma^{\varepsilon}(\bx)z^{\varepsilon}_2 
- (\hat{\gamma} - \bar{\gamma})\bar{I},
&&  \mbox{in } Q_{T},
\label{eq2:infectadosdif}
\\
\nabla z^\varepsilon_1\cdot \bn
&=\nabla z^\varepsilon_2\cdot \bn=0,  
&& \mbox{on }\Gamma,
\label{eq3:fronteradif}
\\
z^\varepsilon_1(\bx,0)&=
z^\varepsilon_2(\bx,0)=0, && \mbox{in }\Omega.
\label{eq4:inicialSdif}
\end{align}
Then, denoting by $(z_{1},z_{2})$ the limit of $(z_{1}^{\varepsilon},z_{2}^{\varepsilon})$ when $\varepsilon\to 0$, from
\eqref{eq1:suceptiblesdif}-\eqref{eq4:inicialSdif}, we deduce that
\begin{align}
(z_1)_t - \Delta z_1 
& = -m\bar{\beta}(\bx)(\bar{S})^{m-1}(\bar{I})^n z_1
-n\bar{\beta}(\bx) (\bar{S})^m(\bar{I})^{n-1} z_2 
&&
\nonumber\\
&
\quad
-(\hat{\beta}-\bar{\beta})(\bar{S})^{m}(\bar{I})^{n} 
+ \bar{\gamma}(\bx)z^{\varepsilon}_2 
+ (\hat{\gamma} - \bar{\gamma})\bar{I},
&&  \mbox{in } Q_{T},
\label{eq1:suceptibleslimit0}
\\
(z_2)_t - \Delta z_2 
& = m\bar{\beta}(\bx)(\bar{S})^{m-1}(\bar{I})^n z_1
+n\bar{\beta}(\bx) (\bar{S})^m(\bar{I})^{n-1} z_2 
&&
\nonumber\\
&
\quad
+(\hat{\beta}-\bar{\beta})(\bar{S})^{m}(\bar{I})^{n} 
- \bar{\gamma}(\bx)z^{\varepsilon}_2 
- (\hat{\gamma} - \bar{\gamma})\bar{I},
&&  \mbox{in } Q_{T},
\label{eq2:infectadoslimit0}
\\
\nabla z_1\cdot \bn
&=\nabla z_2\cdot \bn=0,  
&&\mbox{on }\Gamma,
\label{eq3:fronteralimit0}
\\
z_1(\bx,0)&=
z_2(\bx,0)=0, && \mbox{in }\Omega.
\label{eq4:inicialSlimit0}
\end{align}
Thus, in \eqref{eq:Frechetderivative} we have that
\begin{eqnarray}
\frac{dJ_{\varepsilon}}{d\varepsilon}\Big{|}_{\varepsilon =0}
&= &\int_{\Omega}\Big( \left|S^{\varepsilon}(\cdot,t)-S^{obs} \right|z_1(\cdot,t)+
\left|I^{\varepsilon}(\cdot,t)-I^{obs}\right|z_2(\cdot,t)
 \Big)\,d\bx  \nonumber
\\
&&+\delta\int_{\Omega}\left[ \nabla\bar{\beta}\nabla\left(\hat{\beta}- \bar{\beta}\right) +
\nabla\bar{\gamma}\nabla\left(\hat{\gamma}- \bar{\gamma}\right) \right] \,d\bx\geq 0,
\label{eq:replaceinderivativeFrechet}
\end{eqnarray}
where $(z_{1},z_{2})$ is the solution of 
\eqref{eq1:suceptibleslimit0}-\eqref{eq4:inicialSlimit0}.

On the other hand, from
\eqref{eq10:adjunto1}-\eqref{eq11:adjunto2} and 
\eqref{eq1:suceptibleslimit0}-\eqref{eq2:infectadoslimit0}, we deduce that
\begin{align*}
\frac{\partial}{\partial t}(p_{1}z_{1} + p_{2}z_{2}) =
p_{1}\Delta z_{1} + p_{2}\Delta z_{2} - z_{1}\Delta p_{1} 
- z_{2}\Delta p_{2}
+ (\hat{\beta} - \bar{\beta})\bar{S}^{m}\bar{I}^{n}(p_{2}-p_{1}) 
- (\hat{\gamma} - \bar{\gamma})\bar{I}(p_{2}-p_{1}),
\end{align*}
which implies that
\begin{eqnarray}
\iint_{Q_{T}}\frac{\partial}{\partial t}(p_{1}z_{1} + p_{2}z_{2})d\bx dt =
\iint_{Q_{T}}\left[ (\hat{\beta} - \bar{\beta})\bar{S}^{m}\bar{I}^{n} 
- (\hat{\gamma} - \bar{\gamma})\bar{I} \right](p_{2}-p_{1}) d\bx dt,
\label{eq:integratingonFrechet}
\end{eqnarray}
by integration on $Q_{T}$.
Moreover, we notice that
\begin{align}
&\iint_{Q_{T}}\frac{\partial}{\partial t}(p_{1}z_{1} + p_{2}z_{2})d\bx dt =
\int_{\Omega}\Big( p_{1}(\bx,T)z_{1}(\bx,T) 
+ p_{2}(\bx,T)z_{2}(\bx,T) \Big)\,d\bx
\nonumber\\
&
\hspace{1cm}
= \int_{\Omega}\Big( \left| \bar{S}(\bx,T)-S^{obs}(\bx) \right|z_{1}(\bx,T) +
\left| \bar{I}(\bx,T)-I^{obs}(\bx) \right|z_{2}(\bx,T) \Big)\,d\bx.
\label{eq:integratingonFrechet333}
\end{align}
Then, from \eqref{eq:integratingonFrechet} and \eqref{eq:integratingonFrechet333}
we deduce that
\begin{align}
&\iint_{Q_{T}}\left[ (\hat{\beta} - \bar{\beta})\bar{S}^{m}\bar{I}^{n} 
- (\hat{\gamma} - \bar{\gamma})\bar{I} \right](p_{2}-p_{1}) d\bx dt
\nonumber\\
&
\hspace{2cm}
 = \int_{\Omega}\Big( \left| \bar{S}(\bx,T)-S^{obs}(\bx) \right|z_{1}(\bx,T) +
\left| \bar{I}(\bx,T)-I^{obs}(\bx) \right|z_{2}(\bx,T) \Big)\,d\bx.
\label{eq:integratingonFrechet333_88}
\end{align}
We can conclude the proof of \eqref{teo:necessaryoptimalcondition}
by replacing \eqref{eq:integratingonFrechet333_88} 
in the first term of~\eqref{eq:replaceinderivativeFrechet}.
\end{proof}

\subsection{Some stability results}

\begin{lemma}
\label{lem:unicidad2}
Consider that the sets of functions $\{S,I,p_1,p_2\}$ and 
$\{\hat{S},\hat{I},\hat{p}_1,\hat{p}_2\}$ are solutions
to the systems \eqref{eq1:suceptibles}-\eqref{eq5:inicialI} and
\eqref{eq10:adjunto1}-\eqref{eq14:adjunto5}
with the data $\{\beta,\gamma,S^{obs},I^{obs}\}$ 
and $\{\hat{\beta},\hat{\gamma},\hat{S}^{obs},\hat{I}^{obs}\},$
respectively. Then, there exist the positive constants $\Psi_i,i=1,2,3$ such  that
the estimates
\begin{eqnarray}
&&\|(\hat{S}-S)(\cdot,t) \|_{L^{2}(\Omega)}^{2}
 + \|(\hat{I}-I)(\cdot,t) \|_{L^{2}(\Omega)}^{2} 
 \le
\Psi_1\Big(  \|\hat{\beta}-\beta\|_{L^{2}(\Omega)}^{2} 
+ \|\hat{\gamma}-\gamma\|_{L^{2}(\Omega)}^{2} \Big),
\label{eq:estimation_directo}
\\
&&\|(\hat{p}_1-p_1)(\cdot,t) \|_{L^{2}(\Omega)}^{2}
 + \|(\hat{p}_2-p_2)(\cdot,t) \|_{L^{2}(\Omega)}^{2} 
  \le
\Psi_2\Big(  \|\hat{\beta}-\beta\|_{L^{2}(\Omega)}^{2} 
+ \|\hat{\gamma}-\gamma\|_{L^{2}(\Omega)}^{2} \Big)
\nonumber\\
&& 
\hspace{3cm}
+
\Psi_3\Big(  \|\hat{S}^{obs}-S^{obs}\|_{L^{2}(\Omega)}^{2} 
+ \|\hat{I}^{obs}-I^{obs}\|_{L^{2}(\Omega)}^{2} \Big)
\label{eq:estimation_adjunto}
\end{eqnarray}
holds for any $t\in [0,T]$.
\end{lemma}

\begin{proof}
For the sake of simplicity of the presentation,
we introduce the following notations
\begin{align*}
&\delta S=\hat{S} - S, 
&& \delta p_1=\tilde{p}_{1} - p_{2}, 
&& \delta\beta=\hat{\beta} - \beta
\\
&\delta I=\hat{I} - I, 
&&  \delta p_2=\hat{p}_{2} - p_{2},
&&\delta\gamma=\hat{\gamma} - \gamma.
\end{align*}
Then, from the system
\eqref{eq1:suceptibles}-\eqref{eq5:inicialI} 
for $(S,I)$ and $(\hat{S},\hat{S})$ we have 
that $(\delta S,\delta I)$ satisfy
the system
\begin{align}
(\delta S)_t - \Delta (\delta S) 
& = -\hat{\beta}(\bx)
\Big[(\hat{S})^{m}(\hat{I})^{n}-(S)^{m}(I)^{n}\Big]
&&
\nonumber
\\
&
\hspace{1cm}
-\delta\beta(\bx)(\hat{S})^{m}(\hat{I})^{n}
+\hat{\gamma}(\bx)\delta I+\gamma(\bx) I, 
&&  \mbox{in }\in Q_{T},
\label{eq1:suceptibles:delta}
\\
(\delta I)_t - \Delta (\delta I) 
& = \hat{\beta}(\bx)
\Big[(\hat{S})^{m}(\hat{I})^{n}-(S)^{m}(I)^{n}\Big]
&&
\nonumber
\\
&
\hspace{1cm}
+\delta\beta(\bx)(\hat{S})^{m}(\hat{I})^{n}
-\hat{\gamma}(\bx)\delta I-\gamma(\bx) I, 
&&  \mbox{in }\in Q_{T},
\label{eq2:infectados:delta}
\\
\nabla (\delta S)\cdot \bn
&=\nabla (\delta I)\cdot \bn=0,  
&& \mbox{on }\in\Gamma,
\label{eq3:frontera:delta}
\\
(\delta S)(\bx,0)&=(\delta I)(\bx,0)=0, && 
\mbox{in }\Omega.
\label{eq4:inicial:delta}
\end{align}
Similarly, from the adjoint system \eqref{eq10:adjunto1}-\eqref{eq14:adjunto5},
we deduce that
$(\delta p_1,\delta p_2)$ is solution of the system
\begin{align}
(\delta p_1)_t + \Delta (\delta p_1) 
& = m\hat{\beta}(\bx)(\hat{S})^{m-1}(\hat{I})^{n}(\hat{p_1}-\hat{p_2})
-m\beta(\bx)(S)^{m-1}(I)^{n}(p_1-p_2), 
&&  \mbox{in } Q_{T},
\label{eq1:suceptibles:deltap}
\\
(\delta p_2)_t + \Delta (\delta p_2) 
& = n \hat{\beta}(\bx)\hat{S}^{m}\hat{I}^{n-1}(\hat{p}_1-\hat{p}_2)
-\hat{\gamma}(\bx)(\hat{p}_1-\hat{p}_2)
&&
\nonumber
\\
&
\hspace{1cm}
-n \beta(\bx)S^{m} I^{n-1}(p_1-p_2)
+\gamma(\bx)(p_1-p_2), 
&&  \mbox{in } Q_{T},
\label{eq2:infectados:deltap}
\\
\nabla (\delta p_1)\cdot \bn
&=\nabla (\delta p_2)\cdot \bn=0,  
&&  \mbox{on } \in\Gamma,
\label{eq3:frontera:deltap}
\\
(\delta p_1)(\bx,T)&=
\delta S(\bx,T) - \left( \hat{S}^{obs}(\bx) - S^{obs}(\bx) \right),
&&  \mbox{in } \Omega,
\label{eq4:inicialS:deltap}
\\
(\delta p_2)(\bx,T)&=
\delta I(\bx,T) - \left( \hat{I}^{obs}(\bx) - I^{obs}(\bx) \right),
&&  \mbox{in } \Omega.
\label{eq4:inicialI:deltap}
\end{align}
Then, the proofs of \eqref{eq:estimation_directo}
and \eqref{eq:estimation_adjunto} are reduced to get estimations
for the systems \eqref{eq1:suceptibles:delta}-\eqref{eq4:inicial:delta}
and \eqref{eq1:suceptibles:deltap}-\eqref{eq4:inicialI:deltap}, respectively.

In order to prove \eqref{eq:estimation_directo}, we test
the equations \eqref{eq1:suceptibles:delta} and 
\eqref{eq2:infectados:delta} by $\delta S$ and $\delta I$, respectively.
Then, adding the results we get
\begin{align}
&\frac{1}{2}\frac{d}{dt}
\left(
\|\delta S(\cdot,t)\|_{L^{2}(\Omega)}^{2} 
+\|\delta I(\cdot,t)\|_{L^{2}(\Omega)}^{2} 
\right) 
+ \|\nabla(\delta S)(\cdot,t)\|_{L^{2}(\Omega)}^{2} 
+\|\nabla(\delta I)(\cdot,t)\|_{L^{2}(\Omega)}^{2}
\nonumber\\
&
\quad
\le
\int_{\Omega}|\hat{\beta}(\bx)|
\Big|\hat{S}^{m}\hat{I}^{n}-S^{m}I^{n}\Big||\delta S|\, d\bx
+\int_{\Omega}|\delta\beta(\bx)||\hat{S}|^{m}\hat{I}|^{n}|\delta S|\, d\bx
+\int_{\Omega}|\hat{\gamma}(\bx)||\delta I||\delta S|\, d\bx
\nonumber\\
&
\quad\;
+\int_{\Omega}|\delta\gamma(\bx)| |I||\delta S|\, d\bx
+\int_{\Omega}|\hat{\beta}(\bx)|
\Big|\hat{S}^{m}\hat{I}^{n}-S^{m}I^{n}\Big||\delta I|\, d\bx
+\int_{\Omega}|\delta\beta(\bx)||\hat{S}|^{m}|\hat{I}|^{n}|\delta I|\, d\bx
\nonumber\\
&
\quad\;
+\int_{\Omega}|\hat{\gamma}(\bx)||\delta I|^2\, d\bx
+\int_{\Omega}|\delta\gamma(\bx)| |I||\delta S|\, d\bx
\nonumber\\
&\quad
=\sum_{j=1}^8 I_j,
\label{eq:delta:l2estimate}
\end{align}
where $I_j$ are defined by each term. 
Now, using the Corollary~\ref{corolar:bounded}
to get that
\begin{align}
|\hat{S}^{m}\hat{I}^{n}-S^{m}I^{n}|
&=|\hat{S}^{m}\hat{I}^{n}-\hat{S}^{m}I^{n}+\hat{S}^{m}I^{n}-S^{m}I^{n}|
\nonumber\\
&=\left|\hat{S}^{m}n\int_{I}^{\hat{I}}u^{n-1}du+I^{n}m
\int_{S}^{\hat{S}}u^{m-1}du\right|
\nonumber\\
&\le n|\hat{S}|^{m}\int_{I}^{\hat{I}}\mathbb{I}_m^{n-1}du
+m|I|^{n}\int_{S}^{\hat{S}}\mathbb{S}_m^{m-1}du,
\nonumber\\
&\le n\:\mathbb{S}_M^{m}\mathbb{I}_m^{n-1}|\hat{I}-I|
+m\:\mathbb{S}_m^{m-1}\mathbb{I}_M^n|\hat{S}-S|,
\label{eq:holderianadas}
\end{align}
we proceed to get the appropriate bounds for $I_j$. Indeed,
by the Cauchy inequality and \eqref{eq:holderianadas},  we have 
that $I_1$ can be bounded as follows
\begin{align*}
I_1
&
\le
\frac{n\:\overline{b}}{2}
\mathbb{S}_M^{m}\mathbb{I}_m^{n-1}
\left(
\int_{\Omega}|\delta I|^2\, d\bx
+\int_{\Omega}|\delta S|^2\, d\bx
\right)
+m\:\overline{b}
\mathbb{S}_m^{m-1}\mathbb{I}_M^n
\int_{\Omega}|\delta S|^2\, d\bx.
\end{align*}
In the case of $I_2,I_3$ and $I_4,$ we get
\begin{align*}
I_2&
\le 
\frac{1}{2}
\mathbb{S}_M^m\mathbb{I}_M^n
\left(
\int_{\Omega}|\delta \beta|^2\, d\bx
+\int_{\Omega}|\delta S|^2\, d\bx
\right),
\qquad
I_3\le \frac{\overline{r}}{2}
\left(\int_{\Omega}|\delta I|^2\, d\bx
+\int_{\Omega}|\delta S|^2\, d\bx\right),
\\
I_4&\le\frac{1}{2}
\mathbb{I}_M
\left(
\int_{\Omega}|\delta \gamma|^2\, d\bx
+\int_{\Omega}|\delta S|^2\, d\bx
\right).
\end{align*}
Similarly, we deduce that
\begin{align*}
I_5&
\le
n\:\overline{b}
\mathbb{S}_M^{m}\mathbb{I}_m^{n-1}
\int_{\Omega}|\delta I|^2\, d\bx
+\frac{m\:\overline{b}}{2}
\mathbb{S}_m^{m-1}\mathbb{I}_M^n\left(
\int_{\Omega}|\delta I|^2\, d\bx
+\int_{\Omega}|\delta S|^2\, d\bx
\right),
\\
I_6&
\le 
\frac{1}{2}\mathbb{S}_M^{m}\mathbb{I}_M^{n}
\left(
\int_{\Omega}|\delta \beta|^2\, d\bx
+\int_{\Omega}|\delta I|^2\, d\bx
\right),
\\
I_7&\le \overline{r}\int_{\Omega}|\delta I|^2\, d\bx,
\qquad
I_8\le\frac{1}{2}
\mathbb{I}_M
\left(
\int_{\Omega}|\delta \gamma|^2\, d\bx
+\int_{\Omega}|\delta I|^2\, d\bx
\right).
\end{align*}
Thus, from the estimates of $I_j$ and
\eqref{eq:delta:l2estimate} we have that
\begin{align*}
&\frac{d}{dt}
\left(
\|\delta S(\cdot,t)\|_{L^{2}(\Omega)}^{2} 
+\|\delta I(\cdot,t)\|_{L^{2}(\Omega)}^{2} 
\right) 
+ 2\Big(\|\nabla(\delta S)(\cdot,t)\|_{L^{2}(\Omega)}^{2} 
+\|\nabla(\delta I)(\cdot,t)\|_{L^{2}(\Omega)}^{2} \Big)
\\
&\qquad\qquad
\le 
D_1\Big(
\|\delta S(\cdot,t)\|_{L^{2}(\Omega)}^{2} 
+\|\delta I(\cdot,t)\|_{L^{2}(\Omega)}^{2}\Big)
+D_2\Big(\|\delta \beta\|_{L^{2}(\Omega)}^{2} 
+\|\delta \gamma\|_{L^{2}(\Omega)}^{2}\Big),
\end{align*}
where $D_1=2\hat{C}+\mathbb{I}$ with $\hat{C}$ defined on \eqref{hatCCC} 
and $D_2=\mathbb{S}^{m}_M\mathbb{I}^{n}_M+\mathbb{I}_M$.
Then, applying the Gronwall inequality, we deduce that
\begin{align*}
&
\|\delta S(\cdot,t)\|_{L^{2}(\Omega)}^{2} 
+\|\delta I(\cdot,t)\|_{L^{2}(\Omega)}^{2} 
\le 
e^{D_1T}
\Big(\|\delta S_0\|_{L^{2}(\Omega)}^{2} 
+\|\delta I_0\|_{L^{2}(\Omega)}^{2}\Big)
+
D_2T
\Big(
\|\delta \beta\|_{L^{2}(\Omega)}^{2} 
+\|\delta \gamma\|_{L^{2}(\Omega)}^{2}\Big),
\end{align*}
which implies \eqref{eq:estimation_directo} by using \eqref{eq4:inicial:delta}.

The proof of  \eqref{eq:estimation_directo} is developed as follows.
We can easily prove that the algebraic identity
\begin{align}
&  \hat{\zeta}\;\hat{\mathbb{A}}(\hat{p_1}-\hat{p_2})
-\zeta\;\mathbb{A}(p_1-p_2)
\nonumber\\
&
\quad
=
\Big(\hat{\zeta}-\zeta\Big)\hat{\mathbb{A}}\hat{p_1}
+\zeta\;\Big(\hat{\mathbb{A}}-\mathbb{A}\Big)\hat{p_1}
+\zeta\;\mathbb{A}\delta p_1
-\Big(\hat{\zeta}-\zeta\Big)\hat{\mathbb{A}}\hat{p_2}
-\zeta\;\Big(\hat{\mathbb{A}}-\mathbb{A}\Big)\hat{p_2}
-\zeta\;\mathbb{A}\delta p_2
\label{eq:ident}
\end{align}
is valid. Now, in particular, by selecting 
$(\hat{\zeta},\zeta,\hat{\mathbb{A}},\mathbb{A})
=\Big(\hat{\beta},\beta,m(\hat{S})^{m-1}(\hat{I})^{n},m(S)^{m-1}(I)^{n}\Big)$,
we have that \eqref{eq:ident} implies
that the right hand sides of equation  \eqref{eq1:suceptibles:deltap}
can be rewritten as follows
\begin{align}
&  m\hat{\beta}\;(\hat{S})^{m-1}(\hat{I})^{n}(\hat{p_1}-\hat{p_2})
-m\beta\;(S)^{m-1}(I)^{n}(p_1-p_2)
\nonumber\\
&
\qquad
=
m\delta\beta\;(\hat{S})^{m-1}(\hat{I})^{n}\hat{p_1}
+m\beta\;\Big[(\hat{S})^{m-1}(\hat{I})^{n}-(S)^{m-1}(I)^{n}\Big]\hat{p_1}
\nonumber\\
&
\qquad\quad
+m\beta\;(S)^{m-1}(I)^{n}\delta p_1
-m\delta\beta\;(\hat{S})^{m-1}(\hat{I})^{n}\hat{p_2}
\nonumber\\
&
\qquad\quad
-m\beta\;\Big[(\hat{S})^{m-1}(\hat{I})^{n}-(S)^{m-1}(I)^{n-1}\Big]\hat{p_2}
-m\beta\;(S)^{m-1}(I)^{n}\delta p_2.
\label{eq:eq1:suceptibles:deltap:ld}
\end{align}
Then, by testing \eqref{eq1:suceptibles:deltap}
by $\delta p_1$ and using \eqref{eq:eq1:suceptibles:deltap:ld}, we get
\begin{align*}
&\frac{1}{2}\frac{d}{dt}
\|\delta p_1(\cdot,t)\|^2_{L^2(\Omega)}
=\|\nabla (\delta p_1)(\cdot,t)\|^2_{L^2(\Omega)}
\\
&
\hspace{0.3cm}
+\int_{\Omega}m\delta\beta\;(\hat{S})^{m-1}(\hat{I})^{n}\hat{p_1}\delta p_1d\bx
+\int_{\Omega}
m\beta\;\Big[(\hat{S})^{m-1}(\hat{I})^{n}-(S)^{m-1}(I)^{n}\Big]\hat{p_1}\delta p_1d\bx
\\
&
\hspace{0.3cm}
+\int_{\Omega}m\beta\;(S)^{m-1}(I)^{n}(\delta p_1)^2d\bx
-\int_{\Omega}m\delta\beta\;(\hat{S})^{m-1}(\hat{I})^{n}\hat{p_2}\delta p_1d\bx
\\
&
\hspace{0.3cm}
-\int_{\Omega}
m\beta\;\Big[(\hat{S})^{m-1}(\hat{I})^{n}-(S)^{m-1}(I)^{n}\Big]\hat{p_2}\delta p_1d\bx
-\int_{\Omega}m\beta\;(S)^{m-1}(I)^{n}\delta p_1\delta p_2d\bx.
\end{align*}
From Lemma~\ref{lem:boundedp_i}, Corollary~\ref{corolar:bounded},
by using similar arguments to \eqref{eq:holderianadas}, and the Cauchy inequality
we have that
\begin{align}
&-\frac{1}{2}\frac{d}{dt}
\|\delta p_1(\cdot,t)\|^2_{L^2(\Omega)}
+\|\nabla \delta p_1(\cdot,t)\|^2_{L^2(\Omega)}
\nonumber\\
&
\quad
\le
\max\Big\{P_4,P_5\Big\}
\left\{
m\mathbb{S}_m^{m-1}\mathbb{I}_M^n
\Big(
\|\delta p_1(\cdot,t)\|^2_{L^2(\Omega)}
+\|\delta \beta\|^2_{L^2(\Omega)}
\Big)\right.
\nonumber\\
&
\qquad\qquad
+ 
mn\overline{b}
\mathbb{S}_m^{m-1}\mathbb{I}_m^{n-1}
\Big(
\|\delta p_1(\cdot,t)\|^2_{L^2(\Omega)}
+\|\delta I(\cdot,t)\|^2_{L^2(\Omega)}
\Big)
\nonumber\\
&
\qquad\qquad \left.
+ 
m|m-1|\overline{b}
\mathbb{S}_M^{m-2}\mathbb{I}_M^n
\Big(
\|\delta p_1(\cdot,t)\|^2_{L^2(\Omega)}
+\|\delta S(\cdot,t)\|^2_{L^2(\Omega)}
\Big)\right\}
\nonumber\\
&
\quad\quad 
+\frac{m\overline{b}}{2}
\mathbb{S}_m^{m-1}\mathbb{I}_M^n
\left(3\|\delta p_1(\cdot,t)\|^2_{L^2(\Omega)}+\|\delta p_2(\cdot,t)\|^2_{L^2(\Omega)}
\right).
\label{eq:estimadeltap1}
\end{align}
Now, from \eqref{eq:ident}, by selecting
$(\hat{\zeta},\zeta,\hat{\mathbb{A}},\mathbb{A})
=\Big(\hat{\beta},\beta,n(\hat{S})^m(\hat{I})^{n-1},n(S)^m(I)^{n-1}\Big)$
and 
$(\hat{\zeta},\zeta,\hat{\mathbb{A}},\mathbb{A})
=\Big(\hat{\gamma},\gamma,1,1),$ we can rewritten
the right hand side of equation  \eqref{eq2:infectados:deltap}. Then, 
testing \eqref{eq1:suceptibles:deltap}
by $\delta p_2$ and using similar arguments 
we get a similar estimate to \eqref{eq:estimadeltap1}.
Thus, we have that there exist the positive constants $\tilde{E}_i,$ $i=1,2,3,$ such that
\begin{align*}
&-\frac{d}{dt}\left(
\|\delta p_1(\cdot,t)\|^2_{L^2(\Omega)}+\|\delta p_2(\cdot,t)\|^2_{L^2(\Omega)}
\right)
+2\Big(\|\nabla \delta p_1(\cdot,t)\|^2_{L^2(\Omega)}
+\|\nabla \delta p_2(\cdot,t)\|^2_{L^2(\Omega)}\Big)
\\
&\qquad
\le 
\tilde{E}_1\Big(\|\delta p_1(\cdot,t)\|^2_{L^2(\Omega)}
+\|\delta p_2(\cdot,t)\|^2_{L^2(\Omega)}\Big)
\\
&\qquad\quad
+\tilde{E}_2\Big(\|\delta S(\cdot,t)\|^2_{L^2(\Omega)}
+\|\delta I(\cdot,t)\|^2_{L^2(\Omega)}\Big)
+\tilde{E}_3\Big(\|\delta \beta\|^2_{L^2(\Omega)}+\|\delta \gamma\|^2_{L^2(\Omega)}\Big).
\end{align*}
Applying the estimate \eqref{eq:estimation_directo} 
and rearranging some terms we deduce that
\begin{align*}
-\frac{d}{dt}\left(e^{\tilde{E}_1 t}
\Big[\|\delta p_1(\cdot,t)\|^2_{L^2(\Omega)}+\|\delta p_2(\cdot,t)\|^2_{L^2(\Omega)}\Big]
\right)
\le 
(\tilde{E}_2\Psi_1+\tilde{E}_3)
\Big(\|\delta \beta\|^2_{L^2(\Omega)}+\|\delta \gamma\|^2_{L^2(\Omega)}\Big),
\end{align*}
and integrating on $[t,T]$ we have that
\begin{align*}
e^{\tilde{E}_1 t}\Big[
\|\delta p_1(\cdot,t)\|^2_{L^2(\Omega)}+\|\delta p_2(\cdot,t)\|^2_{L^2(\Omega)}\Big]
&\le
e^{\tilde{E}_1 T}\Big[
\|\delta p_1(\cdot,T)\|^2_{L^2(\Omega)}+\|\delta p_2(\cdot,T)\|^2_{L^2(\Omega)}\Big]
\\
&\quad
+T(\tilde{E}_2\Psi_1+\tilde{E}_3)e^{\tilde{C}_1 T}
\Big(\|\delta \beta\|^2_{L^2(\Omega)}+\|\delta \gamma\|^2_{L^2(\Omega)}\Big).
\end{align*}
Hence, we can deduce \eqref{eq:estimation_adjunto} by application of 
the end condition \eqref{eq4:inicialI:deltap}.
\end{proof}

\section{Proof of Theorem \ref{teo:inverso}}

\noindent
{\it Existence.}
We can prove the existence by considering the standard strategy of a minimizing sequence
and using the appropriate compactness inclusions. Indeed, 
we clearly note that $U_{ad}(\Omega)\not=\emptyset$ and  $J(\beta ,\gamma)$ is bounded for 
any $(\beta ,\gamma)\in U_{ad}(\Omega).$. Then, we can 
consider that $\{(\beta_n,\gamma_n)\}\subset \mathcal{U}$
is a minimizing sequence of $J$.
Then, the 
compact embedding  
$H^{\intpart{d/2}+1}(\Omega)\subset C^\alpha (\Omega)$ for $\alpha\in ]0,1/2]$,
implies that the minimizing sequence  
$\{(\beta_n,\gamma_n)\}$ is bounded in 
the strong topology of $ C^{\alpha}(\overline{\Omega})\times C^{\alpha}(\overline{\Omega})$
for all $\alpha\in ]0,1/2],$ since
there exists a positive constant $C$ (independent of $\beta,\gamma$ and $n$)
such that 
\begin{align*}
\|\beta_n\|_{C^{\alpha}(\overline{\Omega})}+\|\gamma_n\|_{C^{\alpha}(\overline{\Omega})}
\le 
C\Big(\|\beta_n\|_{H^{\intpart{d/2}+1}(\Omega)}
+\|\gamma_n\|_{H^{\intpart{d/2}+1}(\Omega)}\Big),
\quad \forall \alpha\in ]0,1/2].
\end{align*}
Notice  that 
the right hand is bounded by 
the fact that 
$\beta_n,\gamma_n\in H^{\intpart{d/2}+1}(\Omega)$, see 
the definition of $U_{ad}(\Omega)$ given on \eqref{eq:admissible}.
Now, let us denote by $(S_n,I_n)$ the solution of  
the initial boundary value problem \eqref{eq1:suceptibles}-\eqref{eq5:inicialI} 
corresponding to $(\beta_n,\gamma_n)$. Then, by considering the fact that
$\{(\beta_n,\gamma_n)\}$ is belong to
$ C^{\alpha}(\overline{\Omega})\times C^{\alpha}(\overline{\Omega})$
for all $\alpha\in ]0,1/2]$,
by Theorem~\ref{teo:direct_problem}, we have that  
$S_n$ and $I_n$ are  belong to the H\"older space 
$C^{2+\alpha,1+\frac{\alpha}{2}}(\overline{Q}_T)$
and also
$\{(S_n,I_n)\}$ is a  bounded sequence in the strong topology of
$C^{2+\alpha,1+\frac{\alpha}{2}}(\overline{Q}_T)
\times C^{2+\alpha,1+\frac{\alpha}{2}}(\overline{Q}_T)$
for all $\alpha\in ]0,1/2]$. Thus,
the boundedness of the minimizing sequence and the corresponding
sequence $\{(S_n,I_n)\}$, implies that
there exist 
\begin{align*}
(\overline{\beta},\overline{\gamma})\in 
\Big[C^{1/2}(\Omega)\times C^{1/2}(\Omega)\Big]\cap U_{ad}(\Omega),
\qquad
(\overline{S},\overline{T})\in C^{2+\frac{1}{2},1+\frac{1}{4}}(\overline{Q}_T)
\times
C^{2+\frac{1}{2},1+\frac{1}{4}}(\overline{Q}_T),
\end{align*}
and the subsequences  again labeled by  $\{(\beta_n,\gamma_n)\}$
and $\{(S_n,I_n)\}$ such that 
\begin{align}
&\beta_n \to\overline{\beta}, \quad  \gamma_n \to\overline{\gamma}
\quad\mbox{uniformly on }
 C^{\alpha}(\Omega),
\\
& S_n \to\overline{S},\quad  I_n \to\overline{I}
\quad\mbox{uniformly on }C^{\alpha,\frac{\alpha}{2}}(\overline{Q}_T)
\cap 
C^{2+\alpha,1+\frac{\alpha}{2}}(\overline{Q}_T).
\end{align}
Moreover, we can deduce that $(\overline{S},\overline{I})$ is the solution of  
the initial boundary value problem \eqref{eq1:suceptibles}-\eqref{eq5:inicialI}
corresponding to the coefficients $(\overline{\beta},\overline{\gamma}).$
Hence, by Lebesgue's dominated convergence theorem,
the weak lower-semicontinuity of $L^2$ norm,
and the definition of the minimizing sequence, we have that
\begin{align}
J(\overline{\beta},\overline{\gamma})\le \lim_{n\to\infty} J(\beta_n,\gamma_n)=
\inf_{(\beta,\gamma)\in U_{ad}(\Omega)}J(\beta,\gamma).
\end{align}
Then, $(\overline{\beta},\overline{\gamma})$ is a solution of 
\eqref{eq9:OptimizacionEquivalencia} and the prove of existence 
is concluded.

\vspace{0.5cm}
\noindent
{\it Uniqueness.} We prove the uniqueness by using 
adequately the stability result of Lemma~\eqref{lem:unicidad2}
and the necessary optimality condition
of Lemma~\ref{teo:necessaryoptimalcondition}. To be more precise, let us 
consider that the sets of functions $\{S,I,p_1,p_2\}$ and 
$\{\hat{S},\hat{I},\hat{p}_1,\hat{p}_2\}$ are solutions
to the systems \eqref{eq1:suceptibles}-\eqref{eq5:inicialI} and
\eqref{eq10:adjunto1}-\eqref{eq14:adjunto5}
with the data $\{\beta,\gamma,S^{obs},I^{obs}\}$ 
and $\{\hat{\beta},\hat{\gamma},\hat{S}^{obs},\hat{I}^{obs}\},$
respectively.
From Lemma~\ref{teo:necessaryoptimalcondition} and the hypothesis
that $(\beta,\gamma)$ and $(\hat{\beta},\hat{\gamma})$ are solutions
of \eqref{eq9:OptimizacionEquivalencia} we have that
the following inequalities
\begin{align}
&\iint_{Q_{T}}\left[ \left( \overline{\overline{{\beta}}} 
- \beta \right)S^{m} I^{n} - \left(\overline{\overline{{\gamma}}} 
- \gamma\right)I \right]
(p_{2}-p_{1})\,d\bx dt \nonumber
\\
&
\hspace{1cm}
+ \delta\int_{\Omega}\left[ \nabla\beta\nabla\left(\overline{\overline{{\beta}}}
- \beta\right) +
\nabla\gamma\nabla\left(\overline{\overline{{\gamma}}}
- \gamma\right) \right] \,d\bx\geq 0,
\qquad 
\forall (\overline{\overline{{\beta}}},\overline{\overline{{\gamma}}})
\in U_{ad},
\label{eq:opt_cond_1}
\\
&\iint_{Q_{T}}\left[ \left( \underline{\underline{{\beta}}} 
-  \hat{\beta} \right)\hat{S}^{m} \hat{I}^{n} - \left(\underline{\underline{{\gamma}}} 
- \hat{\gamma}\right)\hat{I} \right]
(\hat{p_{2}}-\hat{p_{1}})\,d\bx dt \nonumber
\\
&
\hspace{1cm}
+ \delta\int_{\Omega}\left[ \nabla\hat{\beta}\nabla\left(\underline{\underline{{\beta}}}
- \hat{\beta}\right) +
\nabla\hat{\gamma}\nabla\left(\underline{\underline{{\gamma}}}
-  \hat{\gamma}\right) \right] \,d\bx\geq 0,
\qquad 
\forall (\underline{\underline{{\beta}}},\underline{\underline{{\gamma}}})
\in U_{ad},
\label{eq:opt_cond_2}
\end{align}
are satisfied, respectively. In particular, selecting 
$(\overline{\overline{{\beta}}},\overline{\overline{{\gamma}}})
=(\hat{\beta},\hat{\gamma})$ in \eqref{eq:opt_cond_1} and
$(\underline{\underline{{\beta}}},\underline{\underline{{\gamma}}})
=(\beta,\gamma)$ in \eqref{eq:opt_cond_2}, and adding
both inequalities, we get
\begin{align}
& \delta\left[
 \| \nabla(\hat{\beta}-\beta)\|^2_{L^2(\Omega)} +\| \nabla(\hat{\gamma}-\gamma)\|^2_{L^2(\Omega)}
 \right]
 \nonumber\\
&\hspace{3cm}
 \le\iint_{Q_{T}}\Big| \hat{\beta}- \beta \Big|
 \Big|\hat{S}^{m}\hat{I}^{n}(\hat{p_{2}}-\hat{p_{1}})-S^{m} I^{n}(p_{2}-p_{1})\Big|\,d\bx dt 
 \nonumber\\
&\hspace{3cm} \quad 
 +\iint_{Q_{T}}| \hat{\gamma}- \gamma |
 |\hat{I}(\hat{p_{2}}-\hat{p_{1}})-I(p_{2}-p_{1})|\,d\bx dt:=I_1+I_2.
 \label{eq:estimate_teo}
\end{align}
Now, from \eqref{eq:holderianadas}, \eqref{eq:ident}, 
Corollary~\ref{corolar:bounded}, Lemma~\ref{lem:boundedp_i},
and the Cauchy inequality,
we observe that
\begin{align*}
I_1&\le
\iint_{Q_{T}}| \hat{\beta}- \beta |
|\hat{S}^{m}\hat{I}^{n}-S^{m} I^{n}||\hat{p_{1}}|\,d\bx dt
+\iint_{Q_{T}}| \hat{\beta}- \beta |
|\hat{S}^{m}\hat{I}^{n}-S^{m} I^{n}||\hat{p_{2}}|\,d\bx dt
\\
&\quad
+\iint_{Q_{T}}| \hat{\beta}- \beta ||S^{m} I^{n}||\hat{p_{1}}-p_{1}|\,d\bx dt
+\iint_{Q_{T}}| \hat{\beta}- \beta ||S^{m} I^{n}||\hat{p_{2}}-p_{2}|\,d\bx dt
\\
&\le
\frac{n}{2}\mathbb{S}_M^m \mathbb{I}_m^{n-1}
\max\Big\{P_4,P_5\Big\}
\left(T\| \hat{\beta}-\beta\|^2_{L^2(\Omega)}
+\int_0^T\| \hat{I}(\cdot,t)-I(\cdot,t)\|^2_{L^2(\Omega)}dt
\right)
\\
&\quad
+\frac{m}{2}\mathbb{I}_M^n\mathbb{S}_m^{m-1}
\max\Big\{P_4,P_5\Big\}
\left(T\| \hat{\beta}-\beta\|^2_{L^2(\Omega)}
+\int_0^T\| \hat{S}(\cdot,t)-S(\cdot,t)\|^2_{L^2(\Omega)}dt
\right)
\\
&\quad
+\frac{m}{2}\mathbb{S}_M^m \mathbb{I}_M^n
\left(2T\| \hat{\beta}-\beta\|^2_{L^2(\Omega)}
+\int_0^T\| (\hat{p}_1-p_1)(\cdot,t)\|^2_{L^2(\Omega)}dt
+\int_0^T\| (\hat{p}_2-p_2)(\cdot,t)\|^2_{L^2(\Omega)}dt
\right)
\end{align*}
and
\begin{align*}
 I_2&\le\max\Big\{P_4,P_5\Big\}
 \left(T\| \hat{\gamma}-\gamma\|^2_{L^2(\Omega)}
 +
 \int_0^T\| \hat{I}(\cdot,t)-I(\cdot,t)\|^2_{L^2(\Omega)}dt\right)
 \\
 &\quad
 +\mathbb{I}_M
 \left(T\| \hat{\gamma}-\gamma\|^2_{L^2(\Omega)}
 +
 \int_0^T\|(\hat{p}_1-p_{1})(\cdot,t)\|^2_{L^2(\Omega)}dt\right).
\end{align*}
From Lemma~\ref{lem:unicidad2} and the estimates of  $I_1$ and $I_2$ 
in~\eqref{eq:estimate_teo}  we have that
\begin{align}
& \delta\left[
 \| \nabla(\hat{\beta}-\beta)\|^2_{L^2(\Omega)} +\| \nabla(\hat{\gamma}-\gamma)\|^2_{L^2(\Omega)}
 \right]
 \nonumber\\
&
\quad
\le
\Upsilon_1
\left[
 \| \hat{\beta}-\beta\|^2_{L^2(\Omega)} +\| \hat{\gamma}-\gamma\|^2_{L^2(\Omega)}
 \right]
+
 \Upsilon_2
 \left[
 \| \hat{S}^{obs}-S^{obs}\|^2_{L^2(\Omega)} +\| \hat{I}^{obs}-I^{obs}\|^2_{L^2(\Omega)}
 \right].
\label{eq:previous_inequality}
\end{align}
where
\begin{align*}
\Upsilon_1&=\left[\left(
\frac{n}{2}\mathbb{S}_M^m \mathbb{I}_m^{n-1}
+\frac{m}{2}\mathbb{S}_m^{m-1} \mathbb{I}_M^n+1
\right)(1+\Psi_1)\max\Big\{P_4,P_5\Big\}
+\left(\frac{m}{2}\mathbb{S}_M^m \mathbb{I}_M^n+\mathbb{I}_M\right)(2+\Psi_2)\right]T,
\\
\Upsilon_2&=
  \left(\frac{m}{2}\mathbb{S}_M^m \mathbb{I}_M^n+\mathbb{I}_M\right)
 \Psi_3 T.
\end{align*}
Now, considering that $(\hat{\beta},\hat{\gamma}),(\beta,\gamma)\in\mathcal{U}(\Omega),$
by the generalized Poincar\'e inequality, we have that
\begin{align*}
 &\| \hat{\beta}-\beta\|^2_{L^2(\Omega)} +\| \hat{\gamma}-\gamma\|^2_{L^2(\Omega)}
 \\
 &
 \qquad
 \le 
 C_{poi}\Big(\| \nabla(\hat{\beta}-\beta)\|^2_{L^2(\Omega)} 
 +\| \nabla(\hat{\gamma}-\gamma)\|^2_{L^2(\Omega)}
 +\| \hat{\beta}-\beta\|^2_{L^1(\Omega)} +\| \hat{\gamma}-\gamma\|^2_{L^1(\Omega)}\Big)
 \\
 & \qquad
 = 
 C_{poi}\Big(\| \nabla(\hat{\beta}-\beta)\|^2_{L^2(\Omega)} 
 +\| \nabla(\hat{\gamma}-\gamma)\|^2_{L^2(\Omega)}\Big).
\end{align*}
Then, in \eqref{eq:previous_inequality}
we have that
\begin{align*}
\Big(\delta-\Upsilon_2C_{poi}\Big)\left[
 \| \nabla(\hat{\beta}-\beta)\|^2_{L^2(\Omega)} +\| \nabla(\hat{\gamma}-\gamma)\|^2_{L^2(\Omega)}
 \right]
 \le
 \Upsilon_2
 \left[
 \| \hat{S}^{obs}-S^{obs}\|^2_{L^2(\Omega)} +\| \hat{I}^{obs}-I^{obs}\|^2_{L^2(\Omega)}
 \right].
\end{align*}
Thus, selecting $\Theta=\Upsilon_2C_{poi}$ we deduce the uniqueness
up an additive constant.

\section*{Acknowledgments}

We acknowledge the support of the research projects 
DIUBB GI 172409/C, DIUBB 183309 4/R, Posdoctoral Program, and FAPEI at
Universidad del B{\'\i}o-B{\'\i}o (Chile); 
and CONICYT (Chile) through the program ``Becas de Doctorado''

\end{document}